\documentclass{article}
\usepackage{amsmath,amsxtra,amssymb,amsthm,amsfonts}
\usepackage{graphicx} 
\usepackage{color,graphicx}
\usepackage[margin=1.05in]{geometry}
\usepackage{array}
\usepackage[T1]{fontenc}

\usepackage[pdftex,linktocpage=true,colorlinks,citecolor=blue,linkcolor=blue,pagebackref]{hyperref}
\usepackage{hyperref}
\usepackage{cleveref}

\numberwithin{equation}{section}
\newtheorem{lemma}{Lemma}[section]
\newtheorem{prop}[lemma]{Proposition}
\newtheorem{theorem}[lemma]{Theorem}
\newtheorem{cor}[lemma]{Corollary}

\newtheorem{rem}[lemma]{Remark}
\newtheorem{example}[lemma]{Example}
\newtheorem{definition}[lemma]{Definition}

\newcommand{\C}{\mathbb{C}}
\newcommand\Tr{\mathop{\rm Tr}\nolimits}

\def\ket#1{| #1 \rangle}

\def\kb#1#2{|#1\rangle\!\langle #2 |}

\begin{document}

\title{Quasiorthogonality of Commutative Algebras, Complex Hadamard Matrices, and Mutually Unbiased Measurements}

\author{ 
Sooyeong Kim\textsuperscript{1,2}\footnote{Contact: sooyeong@uoguelph.ca}, David Kribs\textsuperscript{1}, 
	 Edison Lozano\textsuperscript{2}, Rajesh Pereira\textsuperscript{1}, and Sarah Plosker\textsuperscript{2}
}

\maketitle

\begin{abstract}
We deepen the theory of quasiorthogonal and approximately quasiorthogonal operator algebras through an analysis of the commutative algebra case. 
We give a new approach to calculate the measure of orthogonality between two such subalgebras of matrices, based on a matrix-theoretic notion we introduce that has a connection to complex Hadamard matrices. We also show how this new tool can yield significant information on the general non-commutative case. We finish by considering quasiorthogonality for the important subclass of commutative algebras that arise from mutually unbiased bases (MUBs) and mutually unbiased measurements (MUMs) in quantum information theory. We present a number of examples throughout the work, including a subclass that arises from group algebras and Latin squares.

    \medskip

	
\end{abstract}

\addtocounter{footnote}{1}
\footnotetext{Department of Mathematics \& Statistics, University of Guelph, Guelph, ON, Canada  N1G 2W1 N1G 2W1}
\addtocounter{footnote}{1}
\footnotetext{Department of Mathematics \& Computer Science, Brandon University, Brandon, MB, Canada R7A 6A9}

\section{Introduction} \label{sec:intro}

The notion of quasiorthogonality for operator algebras was introduced to provide a quantitative measure of the geometric relationships between algebras, through their conditional expectations and trace inner products \cite{petz2007complementarity,weiner2010}. Pivotal to the development and motivation for considering quasiorthogonality were applications in quantum information theory, and in particular for  investigations of mutually unbiased bases (MUBs) \cite{bandyopadhyay2002MUBs, boykin2005MUBs, pittenger2004} and their associated commutative algebras. The scope of quasiorthogonality related analyses has more recently expanded beyond its origins, with applications in quantum privacy, complementarity and quantum error correction to name a few \cite{kribs2021approximate, kribs2018quantum, levick2016private, levick2017quantum}. 

Although investigations of quasiorthogonality have understandably expanded to general algebras that are not necessarily commutative, it appears that further analysis of the case of commutative algebras is warranted. 
Indeed, this is particularly relevant as, for instance, many fundamental structures in quantum information theory are inherently commutative, notably the maximal abelian subalgebras (MASAs) arising from MUBs. 
Furthermore, the importance of the commutative case has been emphasized by recent developments in the theory of mutually unbiased measurements (MUMs) \cite{farkas2023MUMs,tavakoli2021mutually}. These structures generalize MUBs by relaxing two key constraints: the requirement for rank-one projections and the necessity for matching Hilbert space dimensions. Despite these generalizations, MUMs maintain the crucial property that their associated algebras are commutative, providing an expanded framework for considering quasiorthogonality. The commutative nature of all of these algebras suggests the possibility of stronger results for the subclass than those available in the general non-commutative setting. 

In this paper, we first establish a connection between the quasiorthogonality of commutative algebras and complex Hadamard matrices, which are well-studied 
matrices whose entries have constant modulus and which satisfy $HH^* = nI_n$, where $n$ is the matrix dimension \cite{TZ}. Built on this connection, we introduce a new notion into the analysis, that of a `quasiable' matrix, which allows us to derive a new matrix-theoretic technique to compute the quasiorthogonality measure between pairs of commutative algebras. After some more work, we show how this approach can be extended to the general non-commutative case as well. And, as noted above, we finish by considering quasiorthogonality for the important subclass of commutative algebras that arise from MUMs in quantum information theory. Further details on the contents are given below.     

This paper is organized as follows. In Section~\ref{sec:theory}, we present the mathematical preliminaries, reviewing the notions of conditional expectations, MASAs, complex Hadamard matrices, and the concept of quasiorthogonality along with some of its characterizations. 
In Section~\ref{sec:res}, we present our analysis of the commutative algebra setting. In particular,  we establish  a connection between quasiorthogonal commutative algebras and complex Hadamard matrices, we develop the notion of a quasiable doubly stochastic matrix, and we show a direct link between quasiable  matrices and quasiorthogonal unital commutative algebras. We also consider a subclass that arises from group algebras and the study of Latin squares. In Section~\ref{sec:nc}, we begin by illustrating the differences between commutative and non-commutative settings via the notion of separating vectors for algebras. We then provide a characterization of quasiorthogonal (non-commutative) subalgebras built on the notion of quasiable matrices introduced in the previous section, and we characterize quasiorthogonal algebras via quasiorthogonality of their MASAs. In Section~\ref{sec:examp}, we explore quasiothogonality in detail for a distinguished special case of commutative algebras, those that correspond to MUBs and MUMs in quantum information theory. This section can be read independently of the previous two sections if desired. We feature several examples and forward-looking remarks throughout the work.

\section{Background} \label{sec:theory}
We will work with finite-dimensional complex-valued matrices throughout the paper.  The algebra of all such matrices for fixed dimension $n$ will be denoted $M_n$. We denote the trace of a matrix $X$ by $\mathrm{Tr}(X)$. The {\it Frobenius norm} of a matrix $A\in M_n$ is defined as $\|A\|_F:=\sqrt{\sum_{i,j=1}^n|a_{i,j}|^2}=\sqrt{\Tr(AA^*)}$, where $A^*$ denotes the conjugate transpose of the matrix $A$. (This can also be viewed as a 2-norm on matrices.)  A  {\it unital $*$-subalgebra} of   $M_n$ is a subset $\mathcal{A}\subseteq{M_n}$ that is closed under matrix addition, matrix multiplication, scalar multiplication, and taking adjoints, and contains the identity matrix $I_n$. It is well-known that any unital $*$-subalgebra $\mathcal{A}$ of $M_n$ is a finite-dimensional C$^*$-algebra and is unitarily equivalent to an algebra of the form
$
\bigoplus_{k=1}^d \left({M}_{n_k}\otimes I_{m_k}\right),
$
where $\sum_{k=1}^{d} n_km_k=n$ \cite{davidson}.

The concept of a conditional expectation has historically been a central notion in the theory of operator algebras \cite{umegaki}, and has in more recent years found applications in areas such as quantum information  \cite{church,weiner2010}.  A Hermitian matrix $X = X^*$ is positive semidefinite, denoted $X\succeq 0$, provided  all of its eigenvalues are non-negative.
 Let $\mathcal{A}$ be a unital $*$-subalgebra of $M_n$. Then $\mathcal{E}_{\mathcal{A}}:M_n\rightarrow M_n$ is the unique \emph{conditional expectation channel onto $\mathcal{A}$} provided it satisfies the following four conditions:
\begin{enumerate}
	\item $\mathcal{E}_{\mathcal{A}}(A) = A$ for all $A\in \mathcal{A}$;
	\item $\mathcal{E}_{\mathcal{A}}(X) \succeq 0$ whenever $X\succeq 0$;
	\item $\mathcal{E}_{\mathcal{A}}(A_1 X A_2) = A_1\mathcal{E}_{\mathcal{A}}(X)A_2$ for all $A_1,A_2\in \mathcal{A}$ and $X\in M_n$; and 
	\item $\mathrm{Tr}(\mathcal{E}_{\mathcal{A}}(X)) = \mathrm{Tr}(X)$ for all $X\in M_n$.
\end{enumerate}
It is an easy exercise to see that, for any unital $*$-subalgebra $\mathcal A$,  the conditional expectation channel is the orthogonal projection onto $\mathcal{A}$ where the underlying inner product is the trace inner product: $\langle X,Y\rangle=\Tr(XY^*)$.

Two unital $*$-subalgebras $\mathcal{A}$ and $\mathcal{B}$ of $M_n$ cannot be orthogonal since they both contain the identity. However, they can be thought of as {\it quasi}-orthogonal in a natural way: if $\mathcal A\cap \{I\}^\perp$ and $\mathcal B\cap \{I\}^\perp$ are orthogonal in the trace inner product. 
This motivates the following definition \cite[Definition 2]{levick2016private}. 

\begin{definition}\label{q.o} Two unital $*$-subalgebras $\mathcal{A}$ and $\mathcal{B}$ of $M_n$ are \emph{quasiorthogonal} if they satisfy any one of the following equivalent conditions:
	\begin{enumerate}
		\item $\mathrm{Tr}\bigl((A - \frac{\mathrm{Tr}(A)}{n}I_n)(B-\frac{\mathrm{Tr}(B)}{n}I_n)\bigr) = 0$ for all $A\in \mathcal{A}$, $B\in \mathcal{B}$;
		\item $\mathrm{Tr}(AB) = \frac{\mathrm{Tr}(A)\mathrm{Tr}(B)}{n}$ for all $A\in \mathcal{A}$, $B\in \mathcal{B}$;
		\item\label{3} $\mathcal{E}_{\mathcal{A}}(B) = \frac{\mathrm{Tr}(B)}{n}I_n$ for all $B\in \mathcal{B}$ and $\mathcal{E}_{\mathcal{B}}(A) = \frac{\mathrm{Tr}(A)}{n}I_n$ for all $A \in \mathcal{A}$.
	\end{enumerate}
\end{definition}

As a distinguished special case of interest for our investigation, let $\Delta_n$ be the algebra of $n \times n$ complex diagonal matrices. For any $n\in \mathbb N$ with  $n\geq 2$, there are infinitely many abelian subalgebras of $M_n$ determined by $\Delta_n$; they are of the form $U \Delta_n U^* = \{ UDU^* \, : \, D\in\Delta_n\}$ where $U$ is a unitary matrix, with the centre of the algebra being $\mathbb C I_n$. 
Each of the algebras $U \Delta_n U^*$ is a maximal abelian subalgebra (MASA) inside $M_n$, and by the general form for $\ast$-subalgebras given above, it follows that every MASA of $M_n$ is of this form. 

A {\it complex Hadamard matrix} $H= [h_{i,j}]\in M_n$ is defined as a unimodular matrix whose adjoint is $n$ times its inverse; that is,   $\vert h_{i,j}\vert=1$ for all $1\le i,j\le n$ and  $HH^*=nI_n$.  Of course, any (real) Hadamard matrix is a complex Hadamard matrix.  Complex Hadamard matrices are an important concept in what follows. See \cite{TZ} for a comprehensive survey on this important family of matrices.  


  We have the following characterization of quasiorthogonal MASAs of $M_n$. Although it is a fairly straightforward observation, it does not appear to have been explicitly mentioned in the literature as far as we know. 
  
\begin{prop}\label{thm:cH}   Let $U\in M_n$ be a  unitary matrix.  Then the algebras  $\Delta_n$ and $U\Delta_nU^*$ are quasiorthogonal if and only if $\sqrt{n}U$ is a complex Hadamard matrix.
\end{prop}

\begin{proof}
     Let $U$ be a fixed unitary matrix. Using the fact that the $(i,i)$th entry of $UDU^*$ for any diagonal matrix $D = \mathrm{diag}(d_{ii})\in M_n$ is
$\sum_j |u_{i,j}|^2 d_{j,j}$, we see that $\sqrt{n}U$ is a complex Hadamard matrix if and only if
all the diagonal entries of $UDU^*$ are equal to $n^{-1}\Tr(UDU^*)$ for all diagonal matrices $D\in M_n$. 
\end{proof}

In \cite{weiner2010}, Weiner introduced the following quantitative measure of orthogonality for algebras, which yields quasiorthogonality at its minimum value.

\begin{definition} For unital $*$-algebras $\mathcal{A},\mathcal{B}\subseteq M_n$, the \emph{measure of orthogonality} between them is given by
	\begin{equation}\label{mq1} Q (\mathcal{A},\mathcal{B}):=\mathrm{Tr}(T_{\mathcal{A}}T_{\mathcal{B}})\end{equation} where $T_{\mathcal{A}}$ is the (any) matrix representation of the conditional expectation channel $\mathcal{E}_{\mathcal{A}}$ acting on the vector space $M_n$.
\end{definition}

We note that in general there are a few different ways to compute this measure (see \cite{kribs2021approximate} for a discussion). For instance, if $\{A_i\}_i$ and $\{B_j \}_j$ are orthonormal bases in the trace inner product for, respectively, $\mathcal A$ and $\mathcal B$, then  
\[
    Q (\mathcal{A},\mathcal{B}) = \sum_{i,j} \left| \mathrm{Tr}(A_iB_j)\right|^2 . 
\]
It is easy to see that $Q (\mathcal{A},\mathcal{B}) \geq 1$, as we can always choose orthonormal bases for both algebras that include $\frac{1}{\sqrt{n}} I_n$ (as the algebras are unital). In fact, one can show that the minimum ($Q (\mathcal{A},\mathcal{B}) = 1$) is attained exactly when the algebras are quasiorthogonal (see Corollary~1 of \cite{kribs2021approximate} for a derivation based on the above form). 

We will investigate this quantity in the rest of the paper for special classes of commutative algebras, but note that for the MASAs, it follows from Proposition~\ref{thm:cH} that $Q(\Delta_n, U\Delta_nU^*)$ can  be seen as a measure of how close $\sqrt{n}U$ is to being a complex Hadamard matrix.

\section{Quasiorthogonality of Commutative $*$-Algebras}  \label{sec:res}

Let $\mathcal{A},\mathcal{B}\subseteq M_n$ be unital commutative $*$-algebras. In the commutative setting, the representation theory of unital $*$-subalgebras mentioned in Section~\ref{sec:theory} simplifies to the existence of unitary matrices $U$ and $V$ such that 
\begin{align}\label{eq:commutativecase}
    \mathcal{A} = U\left(\bigoplus_{k=1}^{d_1}\mathbb{C}I_{m_k}\right) U^*\quad \textnormal{and} \quad \mathcal{B} = V\left(\bigoplus_{k=1}^{d_2}\mathbb{C}I_{n_k}\right) V^*, 
\end{align}
where $\sum_{k=1}^{d_1} m_k = n = \sum_{k=1}^{d_2} n_k$. 

We note that $Q(\mathcal A, \mathcal B)$ is unitarily invariant, in that $Q(\mathcal A,\mathcal B)=Q(W{\mathcal A}W^*, W{\mathcal B}W^*)$ for any unitary $W$; indeed, note that $\{W A_i W^*\}_i$ form an orthonormal basis for $\mathcal A$ if $\{A_i\}_i$ do, and then use the formula for $Q(\mathcal A,\mathcal B)$ noted at the end of the last section. So in particular, $Q(\mathcal A,\mathcal B)=Q(V^*{\mathcal A}V, V^*{\mathcal B}V)=Q(\tilde{\mathcal A}, \tilde{\mathcal B})$ with $\tilde{\mathcal A}=W\left(\bigoplus_{k=1}^{d_1}\mathbb{C}I_{m_k}\right)W^*$, $\tilde{\mathcal B}=\bigoplus_{k=1}^{d_2}\mathbb{C}I_{n_k}$, and $W=V^*U$. 
Hence for brevity we can without loss of generality assume one of the algebras under quasiorthogonality consideration is actually a subalgebra of $\Delta_n$. 
However, we keep the reference to $U$ and $V$ in what follows to remind the reader of the explicit connection to the subalgebras. 

Based on the block structure of $\mathcal A$ and $\mathcal B$ in equation~\eqref{eq:commutativecase}, we can obtain bases for the algebras as follows. Consider the partition of columns of $U$ into $\{U_1,\dots,U_{d_1}\}$ so that for $k=1,\dots,d_1$, $U_k$ has $m_k$ columns; and similarly, partition columns of $V$ into $\{V_1,\dots,V_{d_2}\}$ so that for $k=1,\dots,d_2$, $V_k$ has $n_k$ columns. Define $P_i = U_iU_i^*$ for $1\leq i\leq d_1$ and $R_j = V_jV_j^*$ for $1\leq j\leq d_2$. Then $P_i$ and $R_j$ are orthogonal projections with rank $m_i$ and $n_j$, respectively. Moreover, $\left\{\frac{1}{\sqrt{m_i}}P_i \mid 1\leq i\leq d_1\right\}$ and $\left\{\frac{1}{\sqrt{n_j}}R_j \mid 1\leq j\leq d_2\right\}$ are orthonormal bases of $\mathcal{A}$ and $\mathcal{B}$, respectively. 

We can use these decompositions to give a new matrix theoretic technique to compute $Q(\mathcal{A},\mathcal{B})$. For $n\times n$ unitary matrices $U$ and $V$, we let $X$ be the $n \times n$ matrix given by
$$
X=X(U,V) = (U^*V)\circ(\overline{U^*V}),
$$
where $\circ$ is the Schur (entry-wise) matrix product given by $U\circ V = [u_{i,j} v_{i,j}]$ when $U = [u_{i,j}]$ and $V=[v_{i,j}]$. This matrix will be important in the results below, and so we refer to it as the {\it block-index matrix of $\mathcal A$ and $\mathcal B$}. 
Note that $X$ is a doubly stochastic matrix: its entries satisfy $0\leq x_{i,j}\leq 1$, $\sum_{i=1}^nx_{i,j}=\sum_{j=1}^nx_{i,j}=1$. (The Schur product of any unitary with its conjugate is doubly-stochastic.) 

Given partitions $r = (m_1,\dots,m_{d_1})$ and $c = (n_1,\dots,n_{d_2})$ of $n$, we consider the matrix 
$X=X(U,V)$ with its rows partitioned according to $r$ and its columns partitioned according to $c$. Now we define $Y(X;r,c)$ to be the matrix $Y(X;r,c) = [y_{i,j}]$ where $y_{i,j}$ is given by the sum of all entries in $(i,j)$-block in $X$, for $1 \leq i \leq d_1$ and $1\leq j \leq d_2$, multiplied by $\frac{1}{\sqrt{m_in_j}}$. If $X,r,c$ are clear from the context, we write $Y(X;r,c)$ as $Y$. 

We illustrate these constructions with explicit details given in the simple example below. 

\begin{example}
{\rm 
    Let $U=I\in M_3$ and $$ V= \begin{bmatrix}
        0 & 1 & 0\\0 & 0 & 1\\1 & 0 & 0
    \end{bmatrix}.$$
    Then $X = V$. For $r = (2,1)$ and $c = (2,1)$, we have $d_1=2=d_2$ and so $Y(X;r,c)\in M_2$. Also, $m_1=2=n_1$, $m_2=1=n_2$ and the $(i,j)$-blocks of the 3 by 3 matrix $X$ here, for $1 \leq i \leq 2$, $1\leq j \leq 2$,  are:  
       $$
       X_{11} = \begin{bmatrix}
        0 & 1 \\
        0 & 0
    \end{bmatrix} 
    \quad 
           X_{12} = \begin{bmatrix}
        0  \\
        1
    \end{bmatrix} 
    \quad 
               X_{21} = \begin{bmatrix}
        1 & 0 
    \end{bmatrix} 
    \quad 
                   X_{22} = \begin{bmatrix}
        0 
    \end{bmatrix} .
    $$ 
Hence we have 
$$
y_{1,1} = \frac{1}{\sqrt{2^2}} (1+0+0+0) \quad \quad
y_{1,2} = \frac{1}{\sqrt{2}} (0+1) \quad \quad 
y_{2,1} = \frac{1}{\sqrt{2}} (1+0) \quad \quad 
y_{2,2} = \frac{1}{\sqrt{1}} (0) , 
$$
and so 
    $$Y(X;r,c) = \begin{bmatrix}
        \frac{1}{2} & \frac{1}{\sqrt{2}}\\
        \frac{1}{\sqrt{2}} & 0
    \end{bmatrix}.$$
}
\end{example}

We are now in the position to present a formula for the measure of orthogonality between two unital $*$-algebras based on the Frobenius norm of this matrix construction. 

\begin{theorem}\label{thm:Fnorm}
    Let $\mathcal{A},\mathcal{B}\subseteq M_n$ be unital commutative $*$-algebras. Let $U,V$ be the unitary matrices for $\mathcal A$, $\mathcal B$ given in equation~\eqref{eq:commutativecase} and let $X=X(U,V)$ be their block-index matrix, with $r = (m_1,\dots,m_{d_1})$ and $c = (n_1,\dots,n_{d_2})$.
    Then 
    \begin{equation}\label{QFrobenius} 
    Q (\mathcal{A},\mathcal{B})=\|Y(X;r,c)\|_F^2.
    \end{equation} 
\end{theorem}

\begin{proof}
Let us first make the observation that we can write the $(i,j)$-entry of $Y$ as 
\[
y_{i,j} = \frac{1}{\sqrt{m_in_j}}\mathbf{1}^T[(U_i^*V_j)\circ(\overline{U_i^*V_j})]\mathbf{1}, 
\]
where $\mathbf{1}$ is the all-ones vector of appropriate length (here it is $n$-dimensional). 
Thus, using the form for $Q (\mathcal{A},\mathcal{B})$ discussed in the previous section (from \cite[Corollary~1]{kribs2021approximate}), and recalling the orthonormal bases for $\mathcal A,\mathcal B$ constructed above, we calculate: 
\begin{align*}
    Q (\mathcal{A},\mathcal{B}) =&\; \sum_{i,j} \left| \frac{1}{\sqrt{m_in_j}}\mathrm{Tr}(P_iR_j)\right|^2 \\
    =&\; \sum_{i,j} \left|\frac{1}{\sqrt{m_in_j}} \mathrm{Tr}\Big((U_i^*V_j)^*(U_i^*V_j)\Big)\right|^2\\
    =&\;\sum_{i,j}\left(\frac{1}{\sqrt{m_in_j}}\mathbf{1}^T[(U_i^*V_j)\circ(\overline{U_i^*V_j})]\mathbf{1}\right)^2\\
    =&\; \|Y(X;r,c)\|_F^2,   
\end{align*}
and the result follows.  
\end{proof}

Note a special case that is also discussed below: suppose that $X(U,V) = \frac{1}{n}J$, where $J$ is the all-ones matrix. Then $U^*V$ is a complex Hadamard matrix and $Q (\mathcal{A},\mathcal{B}) = 1$ regardless of the choice of $r$ and $c$. 

Let us consider the 2-dimensional case in some detail to illustrate the result.

\begin{example}
{\rm 
    Let $\mathcal{A},\mathcal{B}\subseteq M_2$ be unital commutative $*$-algebras, so each algebra is unitarily equivalent to either $\mathbb{C}\oplus \mathbb{C}$ or $\mathbb{C} I$. Let us consider the more interesting non-scalar ($\mathbb{C}\oplus \mathbb{C}$) case for both algebras.    
    Then, since  $U^*V$ is a $2$ by $2$ complex unitary matrix, it can be written as
    \begin{align*}
         \begin{bmatrix}
            a & b \\
            -e^{\mathrm{i}\tau}\bar{b} & e^{\mathrm{i}\tau}\bar{a}
        \end{bmatrix} \quad\text{for some $a,b\in\C$ with $|a|^2+|b|^2=1$ and $\tau\in\mathbb{R}$.}    
    \end{align*}
    Since $|a|^2 = \cos^2\theta$ and $|b|^2 = \sin^2\theta$ for some $\theta\in\mathbb{R}$, we have
    \begin{align*}
         X = \begin{bmatrix}
            \cos^2\theta & \sin^2\theta \\
            \sin^2\theta & \cos^2\theta
        \end{bmatrix}.    
    \end{align*}
    In this case we have $r =(1,1)$ and $c = (1,1)$, and so $X = Y(X;r,c)$ and  
    \[
    Q(\mathcal{A},\mathcal{B}) = \| Y(X;r,c) \|_2^2 = 2(\cos^4\theta+\sin^4\theta) .
    \]
    Thus we have $1\leq Q(\mathcal A, \mathcal B) \leq 2$ here. Let us find the extremal cases. 
    
    For the maximum, we have $\cos^4\theta+\sin^4\theta = 1$, which is equivalent to (after using basic trigonometric identities) requiring $\sin\theta=0$ or $\cos\theta=0$; in other words, $a=0$ or $b=0$. This case corresponds to the algebras $\mathcal A$ and $\mathcal B$ coinciding, and the unitary either fixes the algebra elements or permutes the entries of the diagonal matrix. This is intuitively what one would expect, that the furthest a pair of algebras from this set can be from being quasiorthogonal is when the algebras coincide. 

    For the minimum, and the case of quasiorthogonal algebras, we have $\cos^4\theta+\sin^4\theta = \frac12$, which is equivalent to $\cos^2\theta = \frac12 = \sin^2\theta$; that is,  $|a| = \frac{1}{\sqrt{2}} = |b|$. Thus, this case requires a certain balance between the combined entries of the unitary operators. For instance, in the case that $U=I$ and $V = \frac{1}{\sqrt{2}} \begin{bmatrix} 1 & 1\\ -1 & 1       
     \end{bmatrix}$, with $\mathcal A = \Delta_2$, yields $\mathcal B = V {\mathcal A} V^*$ given by the algebra of matrices, with $a,b\in \mathbb{C}$, 
     \[
     V \begin{bmatrix} a & 0\\ 0 & b  \end{bmatrix} V^* = \frac12 \begin{bmatrix} a+b & b-a\\ b-a & a+b   
     \end{bmatrix} . 
     \]
     We can connect this observation with our earlier result Proposition~\ref{thm:cH}. Indeed, with $\mathcal A = \Delta_2$, that result tells us $B = V {\mathcal A} V^*$ is quasiorthogonal to $\mathcal A$ if and only if $\sqrt{2} V$ is a Hadamard matrix, which one can readily verify is equivalent to the set of conditions we have obtained above on the matrix entries of the unitary. 
}
\end{example}

Based on these matrix constructions and the previous result, we shall provide a matrix-theoretic condition for two unital, commutative $*$-algebras to be quasiorthogonal. We begin with the following definition.

\begin{definition}
    An $n$ by $n$ doubly stochastic matrix $D$ is said to be \emph{quasiable} with respect to partitions $r = (m_1,\dots,m_{d_1})$ and $c = (n_1,\dots,n_{d_2})$ of $n$ if the matrix $D$ whose rows and columns partitioned according to $r$ and $c$, respectively,  satisfies the condition that for all $1\leq i\leq d_1$, $1\leq j \leq d_2$,  the sum of all entries in the $(i,j)$-block is equal to $\frac{m_in_j}{n}$.
\end{definition}

\begin{rem}
{\rm 
	 We can make some observations on the extreme cases of the property of being quasiable. Let $D\in M_n$ be a doubly stochastic matrix. Then: 
	\begin{itemize}
		\item If $D = \frac{1}{n}J$, then there exists a unitary matrix $U$ with $D = U\circ \overline{U}$ 
        and $D$ is quasiable with respect to any $r$ and $c$.
		\item  Suppose $r=c= (1,\dots,1)$. Then $D$ is quasiable with respect to $r$ and $c$ if and only if $D = \frac{1}{n}J$.
		\item If one of $r$ and $c$ is $(n,0,\dots,0)$, then $D$ is quasiable with respect to $r$ and $c$.
	\end{itemize}
    }
\end{rem}


We now present the main result of this section.

\begin{theorem}\label{thm:commutative and quasi-ortho}
    Let $\mathcal{A},\mathcal{B}\subseteq M_n$ be unital commutative $*$-algebras and let $X=X(U,V)$ be their block-index matrix. 
    Then   $\mathcal{A}$ and $\mathcal{B}$ are quasiorthogonal if and only if $X$ is quasiable with respect to $(m_1,\dots,m_{d_1})$ and $(n_1,\dots,n_{d_2})$.
\end{theorem}

\begin{proof}
    Let $Z = \left[\sqrt{m_in_j}\right]$ for $1\leq i\leq d_1$ and $1\leq j\leq d_2$. Since $\sum_{i,j}m_in_j = n^2$, we have $\|Z\|_F^2 = n^2$. With $Y=Y(X;r,c)$, we find from Theorem~\ref{thm:Fnorm}, and applying the Cauchy-Schwarz inequality for the trace inner-product, that
    \[
    Q(\mathcal{A},\mathcal{B}) = \frac{1}{n^2}\| Y\|_F^2 \| Z\|_F^2  
    \geq  \frac{1}{n^2}\langle Y,Z \rangle^2 .
    \]

Note that equality holds if and only if $Y$ and $Z$ are linearly dependent. But $Y$ and $Z$ are linearly dependent if and only if there exists $k\in\mathbb{C}$ such that $y_{i,j} = k\sqrt{m_in_j}$ for all $i,j$. Using $\sum_{i,j}y_{i,j}\sqrt{m_in_j} = n$ and $\sum_{i,j}m_in_j = n^2$, we have $k = 1/n$. Moreover, $\langle Y,Z \rangle^2 = n^2$. Hence, the case of $Y,Z$ linearly dependent corresponds separately to the case that $X$ is quasiable and to the case that $Q(\mathcal A, \mathcal B)=1$ (and the algebras are quasiorthogonal). This completes the proof.
\end{proof}

We point out a necessary condition for quasiorthogonality that follows from the proof of Theorem~\ref{thm:commutative and quasi-ortho}. 

\begin{cor}\label{cor:zero entry}
    Let $\mathcal{A},\mathcal{B}\subseteq M_n$ be unital commutative $*$-algebras. 
    Given partitions $r$ and $c$ of $n$, if $Y(X;r,c)$ contains an entry equal to $0$, then $\mathcal{A}$ and $\mathcal{B}$ are not quasiorthogonal.
\end{cor}

The following result collects some information about $Q(\mathcal{A},\mathcal{B})$ under certain assumptions from the previous results.

\begin{cor}
    Let $\mathcal{A},\mathcal{B}\subseteq M_n$ be unital commutative $*$-algebras. 
    Given partitions $r$ and $c$ of $n$, the following hold:
    \begin{enumerate}
        \item If $r = c = (1,\dots,1)$ then $Q (\mathcal{A},\mathcal{B})$ is the square of the Frobenius norm of a doubly stochastic matrix.
        \item If one of $r$ and $c$ is $(n,0,\dots,0)$, then $\mathcal{A}$ and $\mathcal{B}$ are quasiorthogonal.
        \item Suppose $U = V$. Then $\mathcal{A}$ and $\mathcal{B}$ are quasiorthogonal if and only if one of $r$ and $c$ is $(n,0,\dots,0)$. In particular, if $r= c = (m_1,\dots,m_{d_1})$, then $Q (\mathcal{A},\mathcal{B}) = d_1$.
    \end{enumerate}
\end{cor}

\begin{proof}
    Statement 1 is straightforward from the above results.

    Suppose that one of $r$ and $c$ is $(n,0,\dots,0)$. Since $X$ is doubly stochastic, $Y$ is the all-ones vector multiplied by $1/\sqrt{n}$. This proves Statement $2$.

    Consider Statement 3. It is straightforward to see that if $r= c = (m_1,\dots,m_{d_1})$, then $Q (\mathcal{A},\mathcal{B}) = d_1$. From Statement 2, it suffices to show that if $r$ and $c$ both are not $(n,0,\dots,0)$, then $\mathcal{A}$ and $\mathcal{B}$ are not quasiorthogonal. Since $U = V$, we have $X(U,V) = I$. So, any entry in $Y$ is either $0$ or $1$, and $Y$ has $n$ ones. If $r$ and $c$ both are not $(n,0,\dots,0)$, then $Y$ contains at least $2n$ entries so that there must be an entry that is $0$. By Corollary~\ref{cor:zero entry}, we obtain our desired result. 
\end{proof}

Given the direct connection between unital commutative $*$-algebras $\mathcal{A}, \mathcal{B}$ being quasiorthogonal and their doubly stochastic block-index matrix $X=(U^*V)\circ(\overline{U^*V})$ being quasiable in  Theorem~\ref{thm:commutative and quasi-ortho}, it will be useful if we can quickly determine whether  $X$ is quasiable. For instance, consider the following matrix:
\begin{align}\label{ex:quasiable but not ortho}
	\begin{bmatrix}
	\frac{1}{3} & \frac{2}{3} & 0 \\
	0 & \frac{1}{3} & \frac{2}{3}\\
	\frac{2}{3}  & 0 & \frac{1}{3}
	\end{bmatrix}.
\end{align}
This matrix is not quasiable with respect to $(1,1,1)$ and $(1,1,1)$, but it is quasiable with respect to $(2,1)$ and $(2,1)$. Therefore, if $(U^*V)\circ(\overline{U^*V})$ is of the form in equation~\eqref{ex:quasiable but not ortho}, then $\mathcal{A}$ and $\mathcal{B}$ are quasiorthogonal. However, as will be shown shortly, there are no unitary matrices $U$ and $V$ such that $(U^*V)\circ(\overline{U^*V})$ is of the form in equation~\eqref{ex:quasiable but not ortho}. 

The following definition helps to clarify the situation.

\begin{definition}
	A \emph{unistochastic} 
	matrix is a doubly stochastic matrix whose entries are the squares of the absolute values of the entries of some unitary matrix. 
\end{definition}

We therefore explore the conditions under which a given doubly stochastic matrix is quasiable and unistochastic. Specifically, given an $n$ by $n$ unistochastic matrix $D$ and partitions of $r$ and $c$ of $n$, under what condition is $D$ quasiable with respect to $r$ and $c$? Resolving this question will offer criteria for determining whether two commutative subalgebras are quasiorthogonal.
However, given a doubly stochastic matrix, determining whether it is unistochastic is a challenging problem. In the case of a 3 by 3 matrix, this has been established, as shown in the following proposition from \cite{au19793} (note that the term \textit{orthostochastic} is used in \cite{au19793} to refer to unistochastic matrices). Using this result, we can see that the doubly stochastic matrix in equation~\eqref{ex:quasiable but not ortho} for instance is not unistochastic.

\begin{theorem}\cite{au19793}
	Let $A = [a_{i,j}]_{i,j=1,2,3}$ be a doubly stochastic 3 by 3 matrix. The following hold:
	\begin{enumerate}
		\item[(i)] If $A$ is unistochastic, then for any $j\neq j'$ and for any $l$, 
		\begin{align}\label{temp:eqn}
		\sqrt{a_{l,j}a_{l,j'}}\leq \sum_{\substack{i=1\\i\neq l}}^3\sqrt{a_{i,j}a_{i,j'}} . 
		\end{align}
		\item[(ii)] Conversely, if there exist $j\neq j'$ such that for any $l$, the inequality \eqref{temp:eqn} holds, then $A$ is unistochastic.
	\end{enumerate}
\end{theorem}

\subsection{Group Algebras and Orthogonal Latin Squares}

The connection between the orthogonality of Latin squares  and commutative C$^*$-algebras was studied in \cite{watatani1994latin}. We describe in this subsection how one can view the measure of orthogonality of unital commutative $*$-algebras coming from group algebras in terms of the orthogonality of the corresponding Latin squares. 


An order $n$ Latin square $\mathcal{L}$   whose entries are $\{1,2,...,n\}$ is an $n$ by $n$ array  consisting of   entries  in $\{1,2,...,n\}$ where each value occurs exactly once in each row and exactly once in each column.  We say that an $n$ by $n$ matrix $A$ has pattern $\mathcal{L}$ if there exist $n$ complex numbers $c_1,c_2,...,c_n$ such that $a_{i,j}=c_k$ whenever the $(i,j)$-entry of $\mathcal{L}$ is $k$.  Let $A(\mathcal{L})$ be the set of all matrices with pattern $\mathcal{L}$. Then $A(\mathcal{L})$  is a commutative $*$-subalgebra of  $(M_n, \circ)$,  the C$^*$-algebra of $n$ by $n$ complex matrices where multiplication is the Schur (entrywise) product instead of ordinary multiplication. We note that in  $(M_n, \circ)$,  the adjoint is entrywise conjugation (without taking the transpose) and  the identity element is $J$, the all-ones matrix.  We can now relate the orthogonality of the Latin squares to the quasiorthogonality of their associated $*$-algebras.

\begin{prop}\label{prop:LS} Let $\mathcal{L}_1$ and $\mathcal{L}_2$ be two $n$ by $n$ Latin squares. Then $\mathcal{L}_1$ and $\mathcal{L}_2$ are orthogonal Latin squares if and only if $A(\mathcal{L}_1)$ and $A(\mathcal{L}_2)$ are quasiorthogonal $*$-subalgebras of  $(M_n, \circ)$.
\end{prop}

\begin{proof}  Let $\mathcal{L}_1$ and $\mathcal{L}_2$ be two $n$ by $n$ Latin squares.  We prove this proposition by calculating the measure of orthogonality between $A(\mathcal{L}_1)$ and $A(\mathcal{L}_2)$.   The element in this C$^*$-algebra $(M_n, \circ)$ represented by the matrix $A$ can be viewed as a linear operator which maps each $n$ by $n$ matrix $X$ to the matrix $A\circ X$. Hence the operator-theoretic trace of this element is $\sum_{i,j=1}^n a_{i,j}$.  

For any $i$, let $P_i$ denote the permutation matrix which has a one in every entry where the symbol of $\mathcal{L}_1$  is $i$ and let $Q_i$ denote the permutation matrix which has a one in every entry where the symbol of $\mathcal{L}_2$  is $i$. Then $\{\frac{1}{\sqrt{n}} P_i\}_{i=1}^{n}$ and $\{\frac{1}{\sqrt{n}} Q_i\}_{i=1}^{n}$ are orthonormal bases of $A(\mathcal{L}_1)$ and $A(\mathcal{L}_2)$ respectively in the trace inner product.  Note that the trace of $P_i\circ Q_j$ is the number of times the symbol of an entry of $\mathcal{L}_1$  is $i$ and the symbol of the corresponding entry of $\mathcal{L}_2$ is $j$.  If we call the set of such entries $S_{i,j}$, then 
\[
Q(A(\mathcal{L}_1), A(\mathcal{L}_2) )=n^{-2}\sum_{i,j=1}^n \vert S_{i,j}\vert^2 .  
\]
Since for any two Latin squares we have $\sum_{i,j=1}^n \vert S_{i,j}\vert=n^2$, it follows from Jensen's inequality that 
\[
1=(n^{-2}\sum_{i,j=1}^n \vert S_{i,j}\vert)^2\le n^{-2}\sum_{i,j=1}^n \vert S_{i,j}\vert^2 , 
\]
with equality if and only if $|S_{ij}|=1$ for all $i,j$, which is equivalent to $\mathcal{L}_1$ and $\mathcal{L}_2$ being orthogonal.
\end{proof}

\begin{rem} 
{\rm 
The measure of orthogonality between $*$-algebras thus gives us a measure of orthogonality between two Latin squares which is interesting in its own right, and invites further investigation.  An obvious question is how small $Q(A(\mathcal{L}_1), A(\mathcal{L}_2) )$ can be without being one. This question would be particularly interesting for six by six Latin squares since no two six by six Latin squares can be orthogonal.
}
\end{rem} 

We note that while $A(\mathcal{L})$ is closed under Schur multiplication, it is not closed in general under matrix multiplication. In fact, it will only be closed under matrix multiplication if the permutation matrices that span $A(\mathcal{L})$ form a finite group $G$.  In this case,  $A(\mathcal{L})$ is also an algebra under matrix multiplication and is isomorphic to $\mathbb{C}[G]$, the group algebra of all (complex) linear combinations of elements of $G$.  We examine this case next.

\begin{example}
{\rm 
	Let $G$ and $H$ be a finite abelian group of order $n$. From standard representation theory, the group algebras $\mathbb{C}[G]$ and $\mathbb{C}[H]$ of $G$ and $H$ over the complex numbers are finite-dimensional, unital, commutative $*$-algebras, and furthermore, $\mathbb{C}[G]\cong U\left(\bigoplus_{k=1}^{n}\mathbb{C}\right) U^*$ and $\mathbb{C}[H]\cong V\left(\bigoplus_{k=1}^{n}\mathbb{C}\right) V^*$ where $U$ and $V$ are the character tables of $G$ and $H$, respectively. 

We note that $\mathbb{C}[G]$ and $\mathbb{C}[H]$ both contain the multiplicative identity, as well as the rank one projection onto the trivial character subspace. Since their intersection is a subspace of dimension greater than one, they cannot be quasiorthogonal.

    We can also prove that $\mathbb{C}[G]$ and $\mathbb{C}[H]$ are not quasiorthogonal via the machinery of quasiability:
	Since $G$ and $H$ are finite abelian groups, $G\cong \mathbb{Z}_{p_1}\times \cdots \times \mathbb{Z}_{p_r}$ and $H\cong \mathbb{Z}_{q_1}\times \cdots \times \mathbb{Z}_{q_s}$ for some $r$ and $s$, where $p_1,\dots,p_r,q_1,\dots,q_s$ are prime powers and $p_1+\cdots +p_r = q_1+\cdots +q_s= n$. Let $F_m$ be the $m\times m$ matrix given by $$F_{m} = \frac{1}{\sqrt{m}}\begin{bmatrix}
	\left(e^{\frac{-2\pi i}{m}}\right)^{jk}
	\end{bmatrix}_{j,k = 0,\dots,m-1}.$$
	Then $F_m$ is the character table of $\mathbb{Z}_m$. Moreover,
	$$U = \bigotimes_{k = 1}^r F_{p_k}\quad\text{and}\quad V = \bigotimes_{k = 1}^s F_{q_k}.$$
	Since $U$ and $V$ both have a row and a column consisting of all ones, $U^*V$ has a row with a single one and zeros elsewhere. Hence, $(U^*V)\circ(\overline{U^*V})$ is not quasiable with respect to $(1,\dots,1)$ and $(1,\dots,1)$. Therefore, $\mathbb{C}[G]$ and $\mathbb{C}[H]$ are not quasiorthogonal.

    
}
\end{example}

\section{From Commutative to General (Non-Commutative) $\ast$-Algebras}\label{sec:nc}

In this section, we show how the techniques and tools we developed above in our consideration of quasiorthogonality for commutative algebras, can be used as an ingredient to investigate the notion for general (non-commutative) algebras. 

We first connect the investigation with a notion from operator algebras \cite{rajeshthesis,pereira2006representing}. Recall that a vector $v$ is a
\emph{separating vector} for an algebra $\mathcal A$ if $Av=0$ (the zero vector) for some $A\in \mathcal A$ implies $A=0$ (the zero operator).

\begin{prop}\label{thm:QOsepvec}
    Let $\mathcal A, \mathcal B$ be unital $\ast$-subalgebras of $M_n$. If $\mathcal A$ has a rank-1 projection and $\mathcal B$ is quasiorthogonal to $\mathcal A$, then $\mathcal B$ must have a separating vector. 
\end{prop}

\begin{proof}
If $\mathcal A$ has a rank-1 projection $P\in \mathcal A$, then we can write $P=vv^*$ for some vector $v\in \mathbb C^n$. Now, if $\mathcal A$ and $\mathcal B$ are quasiorthogonal, then item (2) of Definition~\ref{q.o} reduces, for all $B\in \mathcal B$, to 
\[
v^*Bv= \Tr\big((vv^*\big) B) = \Tr(PB) =  \frac{\Tr(P)\Tr(B)}{n} = \frac{\Tr(B)}{n} . 
\]
So for every positive semidefinite element $M\in \mathcal B$, which we can write as $M=N^*N$ for some $N\in \mathcal B$, item (2) can further be written as $v^*N^*Nv=\Tr(N^*N)/n$, where the right-hand side of this equality is strictly positive if and only if $N$ is not the zero matrix. In particular, $Nv$ cannot be the zero vector for any non-zero element $N\in \mathcal B$. In other words, $v$ is a separating vector for $\mathcal B$, and this completes the proof. 
\end{proof}

\begin{rem}
{\rm 
Proposition~\ref{thm:QOsepvec} highlights an immediate difference confronted when one moves from the commutative to the non-commutative setting. In particular, it was proved in \cite{rajeshthesis,pereira2006representing} that a unital $*$-subalgebra $\mathcal A$ of ${M}_n$ has a separating vector if and only if when $\mathcal A$ is viewed in its unitary equivalence form 
$\mathcal A= U\left(\oplus_{i=1}^q \left({M}_{m_k}\otimes I_{n_k}\right)\right)U^* , 
$
for some unitary $U$, 
we have $m_k\le n_k$ for all $k$ (and $\sum_{k=1}^q m_kn_k=n$).

Clearly, then, many non-commutative algebras (which is any algebra with some $m_k >1$) do not have separating vectors. And by the result above, any such algebra cannot be quasiorthogonal to any algebra that includes a rank-1 projection amongst its elements. This is in contrast to the commutative setting, where for any unital commutative $\ast$-algebra $\mathcal A$, it is not hard to show that one can always find an algebra $\mathcal B$ with rank-1 projections for which $\mathcal A$ is quasiorthogonal to it.  

This observation is perhaps unsurprising given the direct link between quasiorthogonality and quantum privacy exhibited elsewhere \cite{kribs2021approximate}. On this issue specifically, we note private quantum channels and their connection with separating vectors (or equivalently, their connection with trace vectors) was explored in \cite{church}. The connection between quasiorthogonality and quantum privacy is discussed further in Section~\ref{sec:privacy}. 
}
\end{rem}

Let $\mathcal A$ and $\mathcal B$ be unital $*$-subalgebras of $M_n$. In the non-commutative setting, the general representation theory states that there exist unitary matrices $U$ and $V$ such that 
	\begin{align}\label{eq:generalcase}
	\mathcal{A} =  U\left(\oplus_{k=1}^{d_1} \left({M}_{a_k}\otimes I_{m_k}\right)\right)U^*,\quad\quad\mathcal{B} =  V\left(\oplus_{k=1}^{d_2} \left({M}_{b_k}\otimes I_{n_k}\right)\right)V^*
	\end{align}
where $\sum_{k=1}^{d_1} a_km_k=n$ and $\sum_{k=1}^{d_2} b_kn_k=n$, and non-commutativity here simply means some $a_k>1$ and some $b_k>1$. 

Based on the analysis of the previous section, we will construct particular orthonormal bases for $\mathcal{A}$ and $\mathcal{B}$ as follows. 
Consider the partition of columns of $U$ into $U_1,\dots,U_{d_1}$ so that for $k=1,\dots,d_1$, the matrix $U_k$ has $a_km_k$  columns, taken in consecutive order. Similarly, partition columns of $V$ into  $V_1,\dots,V_{d_2}$ so that for $k=1,\dots,d_2$, the matrix $V_k$ has $b_kn_k$  columns, again taken in the natural consecutive order. Next, partition $U_k$ (respectively $V_k$) into $U_{k,1},\dots,U_{k,a_k}$ (respectively $V_{k,1},\dots,V_{k,b_k}$) such that each cell contains $m_k$ (respectively $n_k$)  columns, taken in consecutive order. 

Given $1\leq k\leq d_1$ and $1\leq \ell\leq d_2$,  define 
\begin{align*}
&P(k,p_k,q_k) = U_{k,p_k}U_{k,q_k}^*  \quad\quad\text{ $1\leq p_k, q_k\leq a_k$,}\\
&R(\ell,r_\ell,s_\ell) = V_{\ell,r_\ell}V_{\ell,s_\ell}^*  \quad\quad\quad\quad \text{$1\leq r_\ell, s_k\leq b_\ell$}.
\end{align*}
Then $\|P(k,p_k,q_k)\|^2 = m_k$ and $\|R(\ell,r_\ell,s_\ell)\|^2 = n_\ell$. Generalizing the argument from the commutative algebra case, it follows that 
	\begin{align}\label{eq:general bases}
	\left\{\frac{1}{\sqrt{m_k}}P(k,p_k,q_k) \mid 1\leq k\leq d_1,1\leq p_k, q_k\leq a_k\right\}\quad\text{and}\quad \left\{\frac{1}{\sqrt{n_\ell}}R(\ell,p_\ell,q_\ell) \mid 1\leq \ell\leq d_2,1\leq r_\ell, s_\ell\leq b_\ell\right\}
	\end{align}
are orthonormal bases for $\mathcal{A}$ and $\mathcal{B}$, respectively. We note that $P(k,p_k,q_k)$ is nilpotent for $p_k\neq q_k$ and $R(\ell,p_\ell,q_\ell)$ is nilpotent for $p_\ell\neq q_\ell$. 

We provide a simple example of the construction of an orthonormal basis for a  unital $*$-subalgebra of $M_8$ based on this partitioning below. 

\begin{example}\label{ex:hadamard 2}
{\rm 
	Consider $$\mathcal{A} =  U\left(\oplus_{k=1}^{2} \left({M}_{2}\otimes I_{2}\right)\right)U^* ,$$
	where $$U = \left[\begin{array}{c|c!{\vrule width 1.5pt}c|c}
	U_{1,1} & U_{1,2} & U_{2,1} & U_{2,2}
	\end{array}\right] = \frac{1}{\sqrt{8}}\left[\begin{array}{cc|cc!{\vrule width 1.5pt}cc|cc}
	1 & 1 & 1 & 1 & 1 & 1 & 1 & 1\\
	1 & -1 & 1 & -1 & 1 & -1 & 1 & -1\\
	1 & 1 & -1 & -1 & 1 & 1 & -1 & -1\\
	1 & -1 & -1 & 1 & 1 & -1 & -1 & 1\\
	1 & 1 & 1 & 1 & -1 & -1 & -1 & -1\\
	1 & -1 & 1 & -1 & -1 & 1 & -1 & 1\\
	1 & 1 & -1 & -1 & -1 & -1 & 1 & 1\\
	1 & -1 & -1 & 1 & -1 & 1 & 1 & -1\\
	\end{array}\right].$$
	Then, the following is an orthonormal basis for $\mathcal{A}$: 
	$$\left\{\frac{1}{\sqrt{2}}U_{1,1}U_{1,1}^*,\frac{1}{\sqrt{2}}U_{1,1}U_{1,2}^*,\frac{1}{\sqrt{2}}U_{1,2}U_{1,1}^*,\frac{1}{\sqrt{2}}U_{1,2}U_{1,2}^*,\frac{1}{\sqrt{2}}U_{2,1}U_{2,1}^*,\frac{1}{\sqrt{2}}U_{2,1}U_{2,2}^*,\frac{1}{\sqrt{2}}U_{2,2}U_{2,1}^*,\frac{1}{\sqrt{2}}U_{2,2}U_{2,2}^*\right\}.$$
    }
\end{example}

In order to give a formula for the measure of orthogonality $Q (\mathcal{A},\mathcal{B})$ between two unital $*$-subalgebras $\mathcal A$ and $\mathcal B$ in a manner that builds on the commutative case, we 
partition  $U^*V$ and analyse its submatrices. To this end, let \begin{align}\label{eq:partitions r and c}
	r = (m_1\mathbf{1}_{a_1},m_2\mathbf{1}_{a_2},\dots, m_{d_1}\mathbf{1}_{a_{d_1}})\quad\text{and}\quad c = (n_1\mathbf{1}_{b_1},n_2\mathbf{1}_{b_2},\dots, n_{d_2}\mathbf{1}_{b_{d_2}}).
\end{align}
Consider the matrix $U^*V$ by partitioning its rows according to $r$ and its columns according to $c$ as follows:
\begin{align*}
\begingroup
\renewcommand{\arraystretch}{1.5}
U^*V = \left[\begin{array}{c|c|c!{\vrule width 1.5pt}c|c|c!{\vrule width 1.5pt}c!{\vrule width 1.5pt}c|c|c} U_{1,1}^*V_{1,1} & \cdots & U_{1,1}^*V_{1,b_1} & U_{1,1}^*V_{2,1} & \cdots & U_{1,1}^*V_{2,b_2} & \cdots & U_{1,1}^*V_{d_2,1} & \cdots & U_{1,1}^*V_{d_2,b_{d_2}}\\\hline 
\vdots & \ddots & \vdots & \vdots & \ddots & \vdots & \cdots& \vdots & \ddots & \vdots\\\hline 
U_{1,a_1}^*V_{1,1} & \cdots & U_{1,a_1}^*V_{1,b_1} & U_{1,a_1}^*V_{2,1} & \cdots & U_{1,a_1}^*V_{2,b_2} & \cdots & U_{1,a_1}^*V_{d_2,1} & \cdots & U_{1,a_1}^*V_{d_2,b_{d_2}} \\\noalign{\hrule height 1.5pt} U_{2,1}^*V_{1,1} & \cdots & U_{2,1}^*V_{1,b_1} & U_{2,1}^*V_{2,1} & \cdots & U_{2,1}^*V_{2,b_2} & \cdots & U_{2,1}^*V_{d_2,1} & \cdots & U_{2,1}^*V_{d_2,b_{d_2}} \\\hline
\vdots & \ddots & \vdots & \vdots & \ddots & \vdots & \cdots& \vdots & \ddots & \vdots\\\hline
U_{2,a_1}^*V_{1,1} & \cdots & U_{2,a_1}^*V_{1,b_1} & U_{2,a_1}^*V_{2,1} & \cdots & U_{2,a_1}^*V_{2,b_2} & \cdots & U_{2,a_1}^*V_{d_2,1} & \cdots & U_{2,a_1}^*V_{d_2,b_{d_2}} \\\noalign{\hrule height 1.5pt} \vdots & \ddots & \vdots & \vdots & \ddots & \vdots & \cdots& \vdots & \ddots & \vdots\\ \end{array}\right].
\endgroup
\end{align*}
For $1 \leq k \leq d_1$ and $1\leq \ell \leq d_2$, let $\Gamma_{k,\ell}$ be the multi-set (i.e., elements can be repeated) of submatrices $\{ U_{k,p}^* V_{\ell,q} \, : \, 1\leq p\leq a_k, \,\, 1\leq q\leq b_\ell \}$. That is, $\Gamma_{k,\ell}$ is the collection of submatrices within the bold lines pictured above. We then define 
\[
\gamma_{k,\ell} \equiv  \, \mbox{the sum of the squared trace inner product of all distinct pairs of matrices in}\, \Gamma_{k,\ell}. 
\]
The explicit formula for $\gamma_{k,\ell}$ is given in the proof below. 

The following result extends Theorem~\ref{thm:Fnorm} to the general (i.e., including non-commutative) algebra case. 

\begin{theorem}\label{thm:general quasiorthogonality}
	 Let $\mathcal{A},\mathcal{B}\subseteq M_n$ be unital $*$-algebras. Let $U,V$ be the unitary matrices for $\mathcal A$, $\mathcal B$ given in equation~\eqref{eq:generalcase} and let $X=X(U,V)$ be their block-index matrix. Define $r$ and $c$ as in equation~\eqref{eq:partitions r and c}. Then
	 \begin{align}\label{eq:q-o for general}
	 Q (\mathcal{A},\mathcal{B}) =  \|Y(X;r,c)\|_F^2 + \sum_{k=1}^{d_1}\sum_{\ell=1}^{d_2}\frac{1}{m_kn_\ell}\gamma_{k,\ell}.
	 \end{align}     
	 This implies that $\mathcal{A}$ and $\mathcal{B}$ are quasiorthogonal if and only if the following hold:	
	 \begin{enumerate}
	 	\item[(i)] $X$ is quasiable with respect to $r$ and $c$.
	 	\item[(ii)] For all $1 \leq k \leq d_1$ and $1\leq \ell \leq d_2$, the multi-set of matrices  $\Gamma_{k,\ell}$ forms an orthogonal set in the trace inner product.
	 \end{enumerate}
\end{theorem}

\begin{proof}
	Using the form for $Q (\mathcal{A},\mathcal{B})$ discussed above (from \cite[Corollary~1]{kribs2021approximate}), and recalling the orthonormal bases for $\mathcal A,\mathcal B$ constructed in equation~\eqref{eq:general bases}, we calculate: 
	\begin{align*}\nonumber
	Q (\mathcal{A},\mathcal{B}) 
	=&\; \sum_{\substack{1\leq k\leq d_1\\ 1\leq \ell\leq d_2}}\sum_{\substack{1\leq p_k, q_k\leq a_k\\1\leq r_\ell, s_\ell \leq b_\ell}} \left| \frac{1}{\sqrt{m_kn_\ell}}\mathrm{Tr}\left(P(k,p_k,q_k)R(\ell,r_\ell,s_\ell)\right)\right|^2\\\nonumber
	=&\;\sum_{\substack{1\leq k\leq d_1\\ 1\leq \ell\leq d_2}}\sum_{\substack{1\leq p_k, q_k\leq a_k\\1\leq r_\ell, s_\ell\leq b_\ell}} \left| \frac{1}{\sqrt{m_kn_\ell}}\mathrm{Tr}\left(\left(U_{k,p_k}^*V_{\ell,s_\ell}\right)^*\;
	\left(U_{k,q_k}^*V_{\ell,r_\ell}\right) \right)\right|^2\\\nonumber
	=& \sum_{\substack{1\leq k\leq d_1\\ 1\leq \ell\leq d_2}}\sum_{\substack{1\leq p_k\leq a_k\\1\leq r_\ell\leq b_\ell}} \left| \frac{1}{\sqrt{m_kn_\ell}}\mathrm{Tr}\left(\left(U_{k,p_k}^*V_{\ell,r_\ell}\right)^*\;
	\left(U_{k,p_k}^*V_{\ell,r_\ell}\right) \right)\right|^2\\\nonumber
	&\quad+\sum_{\substack{1\leq k\leq d_1\\ 1\leq \ell\leq d_2}}\sum_{\substack{p_k\ne q_k \;\text{or}\;\\ r_\ell\ne s_\ell}} \left| \frac{1}{\sqrt{m_kn_\ell}}\mathrm{Tr}\left(\left(U_{k,p_k}^*V_{\ell,s_\ell}\right)^*\;
	\left(U_{k,q_k}^*V_{\ell,r_\ell}\right) \right)\right|^2\\
	=&\;  \|Y(X;r,c)\|_F^2 +\sum_{\substack{1\leq k\leq d_1\\ 1\leq \ell\leq d_2}}\frac{1}{m_kn_\ell}\gamma_{k,\ell}, 
	\end{align*}
    and this completes the proof. 
\end{proof}

\begin{rem}
	We note that Theorem~\ref{thm:general quasiorthogonality} can be viewed as a generalization of Theorem~4 of \cite{petz2007complementarity}, specifically for the special case with $d_1 = d_2 = 1$, $a_k = b_k$ and $m_k=n_k$.
\end{rem}

The following is an immediate consequence of the theorem, which gives a test to rule out quasiorthogonality. By the {\it multiplicity} of an element in a multi-set, we simply mean the number of times the element appears in the set. 

\begin{cor}
	Given the hypotheses of Theorem~\ref{thm:general quasiorthogonality}, if there exist $k$ and $\ell$ such that $\Gamma_{k,\ell}$ has an element with multiplicity at least $2$, then $\mathcal{A}$ and $\mathcal{B}$ are not quasiorthogonal.
\end{cor}

Let us illustrate this result for a special case, that also allows us to connect with the Hadamard matrix viewpoint discussed previously.  
Let $\mathcal{A},\mathcal{B}\subseteq M_n$ be unital $*$-algebras as in equation~\eqref{eq:generalcase}. If $U^*V$ is a complex Hadamard matrix, we previously observed that when $\mathcal{A}$ and $\mathcal{B}$ are commutative, they are quasiorthogonal regardless of the ranks of the projections generating $\mathcal{A}$ and $\mathcal{B}$. However, when $\mathcal{A}$ and $\mathcal{B}$ are not commutative, the same result does not hold. In particular, the following example shows that $Q(\mathcal{A},\mathcal{B})$ can be dependent on the number of nilpotent matrices in their orthonormal bases.

\begin{example}
{\rm 
	Let $H_0 = \begin{bmatrix}
	1 & 1\\1 & -1
	\end{bmatrix}$, and $H_n = \begin{bmatrix}
	H_{n-1} & H_{n-1} \\ H_{n-1} & -H_{n-1}
	\end{bmatrix}$ for $n\geq 1$. (Note that the matrix in Example~\ref{ex:hadamard 2} is $H_2$.) Consider $$\mathcal{A}_{p,q} =  U\left(\bigoplus_{k=1}^{2^{n+1-p-q}} \left({M}_{2^p}\otimes I_{2^q}\right)\right)U^*\quad\text{and}\quad \mathcal{B}_{p,q} =  \bigoplus_{k=1}^{2^{n+1-p-q}} \left({M}_{2^p}\otimes I_{2^q}\right)$$
	for some $p,q\geq 1$ with $p+q\leq n$, where $U = \frac{1}{\sqrt{2^{n+1}}}H_n$. Since $H_n$ is a Hadamard matrix, $U$ is quasiable. For $1\leq k,\ell\leq 2^{n+1-p-q}$, we see that $\Gamma_{k,\ell}$ consists of $2^{2p}$ matrices each of which is either $\frac{1}{\sqrt{2^{n+1}}}H_{q-1}$ or $-\frac{1}{\sqrt{2^{n+1}}}H_{q-1}$. 
	Since $\Gamma_{k,\ell}$ has an element with multiplicity more than $1$, by the above corollary, $\mathcal{A}_{p,q}$ and $\mathcal{B}_{p,q}$ are not quasiorthogonal. Indeed, 
	$$\gamma_{k,\ell} = \binom{2^{2p}}{2}\frac{2^{4q}}{2^{2(n+1)}}$$ 
	and therefore,
	\[Q(\mathcal{A}_{p,q},\mathcal{B}_{p,q}) = 1+\sum_{k=1}^{2^{n+1-p-q}}\sum_{\ell=1}^{2^{n+1-p-q}}\frac{\gamma_{k,\ell}}{2^{2q}} = 1+\frac{1}{2^{2p}}\binom{2^{2p}}{2}.\]
    }
\end{example}

Finally, we observe how quasiorthogonality of non-commutative algebras depends on the quasiorthogonality of distinguished commutative subalgebras of the algebras. In particular, note that 
in Theorem~\ref{thm:general quasiorthogonality}, letting 
\[
\mathcal{\tilde{A}} =  U\left(\oplus_{k=1}^{d_1} \left(\Delta_{a_k}\otimes I_{m_k}\right)\right)U^* \quad \mathrm{and} \quad  \mathcal{\tilde{B}} =  V\left(\oplus_{k=1}^{d_2} \left(\Delta_{b_k}\otimes I_{n_k}\right)\right)V^* , 
\]
we have $Q(\mathcal{\tilde{A}},\mathcal{\tilde{B}}) = \|Y(X;r,c)\|_F^2$. 
This suggests the quasiorthogonality of $\mathcal{A}$ and $\mathcal{B}$ requires the existence of a pair of MASAs of $\mathcal{A}$ and $\mathcal{B}$ that are quasiorthogonal.

\begin{prop}\label{prop:quasimax}
	Let $\mathcal{A}$ and $\mathcal{B}$ be unital $*$-subalgebras of ${M}_n$.  Then $\mathcal{A}$ and $\mathcal{B}$ are quasiorthogonal if and only if every pair of maximal commutative $*$-subalgebras of $\mathcal{A}$ and $\mathcal{B}$ are quasiorthogonal.
\end{prop}

\begin{proof}  Suppose every pair of maximal commutative subalgebras of $\mathcal{A}$ and $\mathcal{B}$ are quasiorthogonal.  Let $A$ and $B$ be Hermitian elements of $\mathcal{A}$ and $\mathcal{B}$ respectively.  Then the set of all polynomials in $A$ and $B$ are commutative $*$-subalgebras of $\mathcal{A}$ and $\mathcal{B}$, and each is contained in a MASA of the algebra.  Hence $\mathrm{Tr}(AB) = \frac{\mathrm{Tr}(A)\mathrm{Tr}(B)}{n}$.
Since every matrix is a linear combination of two Hermitian matrices, it follows from the linearity of the trace that $\mathrm{Tr}(AB) = \frac{\mathrm{Tr}(A)\mathrm{Tr}(B)}{n}$ for all $A\in \mathcal{A}$ and $B\in \mathcal{B}$.  Hence $\mathcal{A}$ and $\mathcal{B}$ are quasiorthogonal.  The converse is immediate.  
\end{proof}

\begin{rem} 
{\rm 
This result shows that quasiorthogonality of algebras ultimately depends on the quasiorthogonality of sets of commutative subalgebras, which is obviously of conceptual interest. However, it is not clear that the result is of much practical use, as it, in a sense, hides the calculations required to determine quasiorthogonality, as elucidated in Theorem~\ref{thm:general quasiorthogonality}. 
}
\end{rem}

The next example illuminates Proposition~\ref{prop:quasimax} in more detail, giving algebras  $\mathcal{A}$ and $\mathcal{B}$ that are not quasiorthogonal, but such that there exists a pair of MASAs that are quasiorthogonal. 

\begin{example}
{\rm 
    Let $W$ be a $4\times 4$ unitary matrix with block decomposition given by: $$ W = \begin{bmatrix}
	W_1 & W_2\\ W_3 & W_4 
	\end{bmatrix}, $$ 
	where $W_1,W_2,W_3,W_4$ are $2\times 2$ matrices. Let $\mathcal{A}$ be the subalgebra $W \left({M}_{2}\otimes I_{2}\right)W^*$ and $\mathcal{B}$ be the subalgebra $\Delta_{2}\otimes I_{2}$. Then
	$$\left\{\frac{1}{\sqrt{2}}\begin{bmatrix}
	I_2 & 0\\ 0 & 0
	\end{bmatrix}, \frac{1}{\sqrt{2}}\begin{bmatrix}
	0 & 0 \\ 0 & I_2
	\end{bmatrix}\right\}$$
	is an orthonormal basis of $\mathcal{B}$. 
	
	A MASA of $\mathcal{A}$ is of form $W \left((U\Delta_2U^*)\otimes I_{2}\right)W^*$ for some $2\times2$ unitary matrix $U$. Suppose that $U$ is given by \begin{align*}
	\begin{bmatrix}
	a & b \\
	-e^{\mathrm{i}\tau}\bar{b} & e^{\mathrm{i}\tau}\bar{a}
	\end{bmatrix},\quad\text{for some $a,b\in\C$ with $|a|^2+|b|^2=1$ and $\tau\in\mathbb{R}$.}
	\end{align*}
	Since $I_2$ and $\begin{bmatrix}
	|a|^2-\frac{1}{2} & -e^{-i\tau}ab \\ -e^{i\tau}\overline{ab} & |b|^2-\frac{1}{2}
	\end{bmatrix}$ form an orthogonal basis of $U\Delta_2U^*$, it follows that the corresponding maximal commutative algebra $\tilde{\mathcal{A}}$ has an orthonormal basis
	$$\left\{\frac{1}{2}I_4, \frac{1}{\sqrt{2}}W\left(\begin{bmatrix}
	|a|^2-\frac{1}{2} & -e^{-i\tau}ab \\ -e^{i\tau}\overline{ab} & |b|^2-\frac{1}{2}
	\end{bmatrix}\otimes I_2\right)W^*\right\}.$$
	Therefore,  \begin{align*}
	&Q(\tilde{\mathcal{A}},\mathcal{B})\\ =&\; 1+\frac{1}{4}\left[\left(|a|^2-\frac{1}{2}\right)\mathrm{Tr}(W_1^*W_1)+\left(|b|^2-\frac{1}{2}\right)\mathrm{Tr}(W_2^*W_2)+\mathrm{Tr}(e^{-i\tau}abW_1^*W_2)+\overline{\mathrm{Tr}(e^{-i\tau}abW_1^*W_2)}\right]^2 \\
	&+\frac{1}{4}\left[\left(|a|^2-\frac{1}{2}\right)\mathrm{Tr}(W_3^*W_3)+\left(|b|^2-\frac{1}{2}\right)\mathrm{Tr}(W_4^*W_4)+\mathrm{Tr}(e^{-i\tau}abW_3^*W_4)+\overline{\mathrm{Tr}(e^{-i\tau}abW_3^*W_4)}\right]^2.
	\end{align*}
	
	Suppose we take $$ W = \begin{bmatrix}
	W_1 & W_2\\ W_3 & W_4
	\end{bmatrix} = \frac{1}{2}\left[\begin{array}{cc|cc}
	1 & 1 & 1 & 1\\
	1 & -1& 1 & -1\\\hline
	1 & 1 & -1 & -1\\
	1 & -1 & -1 & 1
	\end{array}\right].$$
	Then, $$Q(\tilde{\mathcal{A}},\mathcal{B}) = 1+\frac{(e^{-i\tau}ab+\overline{e^{-i\tau}ab})^2}{2}= 1+2\left(\mathrm{Re}(e^{-i\tau}ab)\right)^2\leq 1+2|e^{-i\tau}ab|^2= 1+2|ab|^2\leq 1+\frac{1}{2}.$$
    If $U$ is taken as a permutation matrix, then $Q(\tilde{\mathcal{A}},\mathcal{B}) = 1$; and if $U$ is taken for some $a,b\in\mathbb{C}$ and $\tau\in\mathbb{R}$ with $|a|=|b|=\frac{1}{2}$ and $\mathrm{Re}(e^{-i\tau}ab) = \frac{1}{2}$, then $Q(\tilde{\mathcal{A}},\mathcal{B}) = \frac{3}{2}$.
    }
\end{example}

\section{Mutually Unbiased Measurements and Quasiorthogonality} \label{sec:examp}

In this section, we consider the special case of commutative algebras generated by sets of projections of the same rank and cardinality, a case with important links to quantum information. Specifically, we will consider families of projections $\{ P_a \}_{a=1}^d$ on $\mathbb C^n$ of rank-$k$ with $\sum_a P_a = I_n$ (so $n=dk$ and the projections have mutually orthogonal ranges), and the commutative algebras they generate: 
\[
\mathcal A = \mathrm{span}\, \{ P_a \, : \, 1 \leq a \leq d  \} \cong (\mathbb C I_k)^{(d)}.
\]
We will use the notation $\{ Q_b \}_{b=1}^d$ and $\mathcal B = \mathrm{span}\, \{ Q_b \, : \, 1 \leq b \leq d  \}$ for a second family of such projections and the algebra they generate. 

The special case of rank-1 projections $P_a = \kb{\psi_a}{\psi_a}$ and $Q_b = \kb{\phi_b}{\phi_b}$ (and $d=n$), which from the algebra perspective is the case of MASA's inside $M_n$, is one of the initial motivating cases for considering the notion of quasiorthogonal algebras \cite{petz2007complementarity,weiner2010}. Indeed, the corresponding algebras $\mathcal A$ and $\mathcal B$ are quasiorthogonal if and only if the states $\{\ket{\psi_a}\}_a$ and $\{\ket{\phi_b}\}_b$ form a pair of {\it mutually unbiased bases} (MUB). That is, for all $1 \leq a,b \leq n$, 
\[
\frac1n = | \langle \psi_a \, | \, \phi_b \rangle |^2 = \Tr ( P_a  Q_b )  = \frac1n \Tr ( P_a ) \Tr (  Q_b ) , 
\]
with the trace identities added to note how the notion is linked with quasiorthogonality. 

More generally, when the projections are higher-rank (so $k> 1$), we have $d$-outcome (von Neumann) measurements (i.e., projections with mutually orthogonal ranges that sum to the identity operator) and their associated algebras. For such families of measurements, the mutually unbiased condition has recently been generalized and investigated \cite{farkas2023MUMs,tavakoli2021mutually}, formally defined as follows.
We also note other related work \cite{kalev2014mutually, kribs2025generalized, pittenger2004, wootters1989optimal}

\begin{definition}\label{def:MUM}
Two $d$-outcome measurements $\mathcal P = \{P_a\}_{a=1}^d$ and $\mathcal Q = \{Q_b\}_{b=1}^d$ acting on $\mathbb C^n$ are called {\rm mutually unbiased measurements (MUMs)} if
\begin{equation}\label{mumeqn}
P_aQ_bP_a = \frac1d P_a \quad \text{ and } \quad  Q_bP_aQ_b = \frac1d Q_b,
\end{equation}
for all $1\leq a,b \leq d$.   
\end{definition}

For the case of rank-1 projections, equation~(\ref{mumeqn}) is simply a restatement of the MUB condition as an operator relation. Given the observation above for MUBs, it is natural to ask if we can use quasiorthogonality to characterize MUMs? As we show in the following result, the answer to this question is: not quite. 

\begin{theorem}\label{quasimums}
    Let $d,k\geq 1$ be fixed positive integers and let $n=dk$. Suppose that $\mathcal P = \{P_a\}_{a=1}^d$ and $\mathcal Q = \{Q_b\}_{b=1}^d$ are families of rank-$k$ projections defining $d$-outcome measurements on $\mathbb C^n$.  Then the following conditions are equivalent: 
    \begin{itemize}
        \item [$(i)$] $\mathcal P = \{P_a\}_{a=1}^d$ and $\mathcal Q = \{Q_b\}_{b=1}^d$ form a MUM. 
         \item [$(ii)$] The algebras $\mathcal A = \mathrm{span}\, \{ P_a \}$ and $\mathcal B = \mathrm{span}\, \{ Q_b \}$ are quasiorthogonal and the operator products $\{ P_a Q_b P_a \}_{a,b}$, respectively $\{ Q_b P_a Q_b \}_{a,b}$, belong to $\mathcal A$, respectively belong to $\mathcal B$.  
    \end{itemize}
\end{theorem}

\begin{proof}
    If $\mathcal P$, $\mathcal Q$ form a MUM, then trivially the operator products in condition $(ii)$ belong to $\mathcal A$ and $\mathcal B$ by equation~(\ref{mumeqn}). To see that the algebras are quasiorthogonal, first note that for all $a,b$, 
    \[
    \Tr (P_a Q_b) = \Tr(Q_b P_a) = \Tr(P_a Q_b P_a) = \frac1d \Tr(P_a) = \frac{k}{d}.  
    \]
    So using the fact that $\{\frac{1}{\sqrt{k}} P_a\}_{a=1}^d$ and $\{\frac{1}{\sqrt{k}} Q_b\}_{b=1}^d$ are orthonormal bases (in the trace inner product) for the algebras, we can calculate the measure of orthogonality between $\mathcal A$ and $\mathcal B$ as follows:   
\[
Q(\mathcal{A},\mathcal{B}) =  \sum_{a,b=1}^{d} |\Tr\Big( \Big( \frac{1}{\sqrt{k}}P_a \Big) \Big( \frac{1}{\sqrt{k}} Q_b \Big) \Big)|^{2} = \frac{1}{k^2}\sum_{a,b=1}^{d} |\Tr(P_a Q_b)|^2
    = 1. 
\]
Hence, $Q(\mathcal{A},\mathcal{B}) = 1$ and the algebras are quasiorthogonal. 

To prove the converse direction we will make use of the conditional expectation description of quasiorthogonality. If the algebras are quasiorthogonal and $\mathcal E_{\mathcal A}$ is the trace-preserving conditional expectation of $M_n$ onto $\mathcal A$, then for all $a,b$ we have 
\[
\mathcal E_{\mathcal A} (P_a Q_b P_a ) = P_a \, \mathcal E_{\mathcal A} (Q_b) \, P_a = P_a \Big( \frac{\Tr(Q_b)}{n} I_n \Big) P_a = \frac{k}{n} P_a = \frac1d P_a. 
\]
However, if we also know the operator $P_a Q_b P_a$ belongs to $\mathcal A$, then we have 
\[
P_a Q_b P_a = \mathcal E_{\mathcal A} (P_a Q_b P_a )  = \frac1d P_a, 
\]
and so the MUM projection equation is satisfied. 
An analogous argument shows that $Q_bP_aQ_b = \frac1d Q_b$ for all $a,b$ as well. Hence $\mathcal P$ and $\mathcal Q$ form a MUM and this completes the proof. 
\end{proof}

\begin{rem}
{\rm     We note that we can generalize the notion of MUMs and obtain a broader result. In Definition~\ref{def:MUM}, by dropping the requirement that the projections in $\mathcal{P}$ and $\mathcal{Q}$ have the same rank and modifying equation~\eqref{mumeqn} so that $P_aQ_bP_a = \frac{\mathrm{Tr}(Q_b)}{n}P_a$ and $Q_bP_aQ_b = \frac{\mathrm{Tr}(P_a)}{n}Q_b$ for all $1\leq a\leq d_1$ and $1\leq b\leq d_2$, where $\mathcal{P}$ and $\mathcal{Q}$ consist of $d_1$ and $d_2$ projections respectively, we can derive an analogous result to Theorem~\ref{quasimums}. However, we have left the result as stated, as MUMs still appear to be the most important class of such projection families that satisfy the relations. 
}
\end{rem}


As discussed above, we note that in the rank-1 projection ($k=1$) case, the operator product requirement in condition $(ii)$ of the theorem is trivially satisfied when the algebras are quasiorthogonal. This is not the case though for higher-rank projections. Indeed, as the following example points out this condition is required for the general theorem. 

\begin{example}\label{MUMcountereg}
{\rm 
Let $\{ \ket{\psi_{a,b}} \, : \, 1 \leq a,b \leq 2  \}$  be an orthonormal basis for $\mathcal H = \mathbb C^4$. (To view this example in terms of qubits one can identify this basis with the two-qubit computational basis $\{ \ket{00}, \ket{01}, \ket{10}, \ket{11} \}$ in the obvious way.) We define two 2-outcome measurements on $\mathcal H$ made up of rank-2 projections as follows: 
\[
    P_a = \sum_{b=1}^2 \kb{\psi_{a,b}}{\psi_{a,b}} \quad  \mathrm{for} \,\,a=1,2 
\]
\[
    Q_b = \sum_{a=1}^2 \kb{\psi_{a,b}}{\psi_{a,b}} \quad \mathrm{for} \,\, b=1,2 .   
\]
Let $\mathcal A = \mathrm{span}\, \{ P_1, P_2 \}$ and $\mathcal B = \mathrm{span}\, \{ Q_1, Q_2 \}$  be the algebras generated by the two projection sets, both of which are unitarily equivalent to $\mathbb C I_2 \oplus \mathbb C I_2$. 

Observe that $\mathcal A$ and $\mathcal B$ are quasiorthogonal. Indeed, the trace-preserving conditional expectation for $\mathcal A$ is given by 
\[
\mathcal E_{\mathcal A} = \mathcal P_1 \circ \mathcal D_1 \circ \mathcal P_1 + \mathcal P_2 \circ \mathcal D_2 \circ \mathcal P_2, 
\]
where $\mathcal P_a (\cdot) = P_a (\cdot ) P_a$ is the compression map, for $a=1,2$, and $\mathcal D_a$ is the completely depolarizing channel on $P_a\mathcal H$ (with Kraus operators given by Pauli operators for $M_2$ represented in each of the blocks that belong to the commutant $\mathcal A^\prime\cong M_2 \oplus M_2$). In particular, for $b=1,2$ we have 
\[
\mathcal E_{\mathcal A} (Q_b) = \frac12 P_1 + \frac12 P_2 = \frac12 I_4 = \frac{\Tr(Q_b)}{4} I_4,  
\]
and similarly for all elements of $\mathcal B$. An analogous calculation follows for $\mathcal E_{\mathcal B}$ and $\mathcal A$, and so the algebras are quasiorthogonal as claimed. 

However, also observe that 
\[
P_a Q_b = Q_b P_a = \kb{\psi_{a,b}}{\psi_{a,b}} , 
\]
and so for all $a,b$, 
\[
P_a Q_b = P_a Q_b P_a \neq \frac12 P_a \quad \mathrm{and} \quad Q_b P_a = Q_b P_a Q_b  \neq \frac12 Q_b. 
\]
Hence these measurements do not form a MUM even though the algebras are quasiorthogonal. 

As uncovered in the theorem above, the issue here is that the products $P_a Q_b$ are rank-1 projections that do not belong to $\mathcal A$ or $\mathcal B$. In this case each pair of  projections $P_a$, $Q_b$ commute and so $P_aQ_bP_a = P_aQ_b = Q_bP_a = Q_b P_a Q_b$, which is not a scalar multiple of either $P_a$ or $Q_b$ as noted above. 
}
\end{example}

As in the MUB case, we can consider approximate versions of MUMs and use the quasiorthogonality quantity $Q(\mathcal A, \mathcal B)$ to measure closeness to an ideal MUM. We can start with one of the standard notions of approximate MUB and try to generalize it. For instance, given two orthonormal bases $\{\ket{\psi_a}\}_a$ and $\{\ket{\phi_b}\}_b$ for a common Hilbert space, we can consider the following inequality satisfied by all vectors in the two bases:
\[
| \langle \psi_a \, | \, \phi_b \rangle |^2 \leq \frac{1+\epsilon}{n }. 
\]
Towards formulating an appropriate notion of approximate MUM that generalizes this approximate MUB description, we expect there are multiple possibilities (as there are in the MUB case), but if we do so by taking motivation from the ideal case, we end up with the following. 

\begin{definition}\label{approxmumdefn}
    Let $\epsilon >0$. We say that two $d$-outcome measurements $\mathcal P = \{P_a\}_{a=1}^d$ and $\mathcal Q = \{Q_b\}_{b=1}^d$ acting on a Hilbert space $\mathcal H$ are {\rm approximate mutually unbiased measurements (AMUM)} if
\begin{equation}\label{amumeqn}
P_aQ_bP_a \leq \frac{1+\epsilon}{d} P_a \quad \text{ and } \quad  Q_bP_aQ_b \leq \frac{1+\epsilon}{d} Q_b,
\end{equation}
for all $1\leq a,b \leq d$. In this case, we will say the measurements are {\rm $\epsilon$-approximate}.    
\end{definition}

There are numerous constructions of approximate MUBs (AMUBs) found in the literature (for instance see
\cite{klappenecker2005approximately,shparlinski2006constructions}). Here we will not consider explicit AMUM constructions, we will just note that every AMUB defines a family of AMUMs with the definition above simply by tensoring the rank-1 states from the AMUB with the identity operator of any fixed size. For instance, if $\{ \ket{\psi_a} \}$ and $\{ \ket{\phi_b} \}$ are an AMUB on $\mathcal H$ with approximation factor $\epsilon >0$, then for any $k\geq 1$, we can define an $\epsilon$-approximate AMUM on $\mathcal H \otimes \mathbb C^k$ with projections 
\[
P_a = \kb{\psi_a}{\psi_a} \otimes I_k \quad \mathrm{and} \quad Q_b = \kb{\phi_b}{\phi_b} \otimes I_k. 
\]
One can easily verify that equation~(\ref{amumeqn}) holds for these projections. 
Indeed, this is also what happens in the ideal ($\epsilon = 0$) case and what motivates our definition above, where one of the standard techniques (so far) to construct a MUM is to take a MUB and `ampliate' it by tensoring all the rank-1 projections in the MUB with the $k$-dimensional identity operator \cite{farkas2023MUMs}. 

The approximate MUM definition is readily seen to encode approximate quasiorthogonality as follows. 

\begin{prop}\label{aquasimums}
    Let $\epsilon > 0$ and let $d,k\geq 1$ be fixed positive integers and let $n=dk$. Suppose that $\mathcal P = \{P_a\}_{a=1}^d$ and $\mathcal Q = \{Q_b\}_{b=1}^d$ are families of rank-$k$ projections defining $d$-outcome measurements on $\mathbb C^n$. If $\mathcal P$ and $\mathcal Q$ are an $\epsilon$-approximate MUM, then the algebras $\mathcal A$ and $\mathcal B$ generated by their projections satisfy: 
    \[
    Q(\mathcal A, \mathcal B) \leq (1 + \epsilon)^2, 
    \]
    with $\epsilon = 0$ if and only if $\mathcal P$ and $\mathcal Q$ form a MUM. 
\end{prop}

\begin{proof}
As in the proof above, we can use the formula for $Q(\cdot,\cdot)$ to obtain: 
\begin{eqnarray*}
Q(\mathcal{A},\mathcal{B}) &=&  \sum_{a,b=1}^{d} |\Tr\Big( \Big( \frac{1}{\sqrt{k}}P_a \Big) \Big( \frac{1}{\sqrt{k}} Q_b \Big) \Big)|^{2} \\ 
&=& \frac{1}{k^2}\sum_{a,b=1}^{d} |\Tr(P_a Q_b P_a)|^2 \\ 
&\leq& \frac{(1+\epsilon)^2}{d^2k^2}\sum_{a,b=1}^{d} |\Tr(P_a)|^2 
= (1 + \epsilon)^2, 
\end{eqnarray*} 
where we have used Eq.~(\ref{amumeqn}) and the fact that the trace is a positive functional for the inequality. 
\end{proof}

\subsection{Connection with Quantum Privacy}\label{sec:privacy}

The approximate quasiorthogonality of algebras was linked with a notion of `relative quantum privacy' for algebras in \cite{kribs2021approximate}. This in turn built on a collection of related notions in quantum information theory called private quantum channels, private quantum codes, private subsystems, and private algebras \cite{ambainis2000private,bartlett2004decoherence,boykin2003optimal,church,crann2016private,jochym2013private,kretschmann2008complementarity,kribs2018quantum,levick2016private,levick2017quantum}. In equation form, the basic idea is a set of states belonging to a set with some structure (e.g., a quantum code, or an algebra) is mapped by a given quantum channel $\Phi$ to a fixed state; i.e., there is some density operator $\rho_0$ such that $\Phi(\rho) = \rho_0$ for all $\rho$ in the set. 

In the context of algebras, two unital subalgebras $\mathcal A$ and $\mathcal B$ of $M_n$ are said to be {\it $\epsilon$-private relative to each other} for some $\epsilon > 0$ if 
\begin{equation}\label{epsprivate}
|| ( \mathcal E_{\mathcal A}  - \mathcal D_n ) \circ \mathcal E_{\mathcal B} ||_2 \leq \epsilon ,
\end{equation}
where $\mathcal D_n: M_n(\C) \rightarrow M_n(\C)$ is the complete depolarizing channel $\mathcal D_n (X) = \frac{\mathrm{Tr}(X)}{n} I_n$,  $|| A ||_2 = ( \mathrm{Tr} (A^* A))^{\frac12}$, and $|| \Phi ||_2 = \sup_{||X||_2 = 1} ||\Phi(X)||_2$ for linear maps $\Phi$.
As a consequence of Theorem~3.3 from \cite{kribs2021approximate}, this notion is symmetric in the relativeness of the algebras (this follows because it is linked with approximate quasiorthogonality of the algebras, which is noted more explicitly in the proof below). 

Observe that the case $\epsilon = 0$ corresponds to one algebra being `privatized' by the trace-preserving conditional expectation onto the other algebra; that is, $\mathcal E_{\mathcal A}(B) = \frac{\mathrm{Tr}(B)}{n} I_n$ for all $B\in \mathcal B$. In this sense, the algebra $\mathcal A$ (or at least its conditional expectation) destroys all (classical and/or quantum) information encoded into the algebra $\mathcal B$, and then small $\epsilon > 0$ cases correspond to an approximate version of this notion. Observe further that the $\epsilon = 0$ case corresponds to quasiorthogonality of the algebras, and the main result of \cite{kribs2021approximate} shows exactly how their approximate quasiorthogonality (measured in the closeness of $Q(\mathcal A, \mathcal B)$ to 1) corresponds to approximate relative privacy of the algebras (measured in the 2-norm above). 
 
We can consider these approximate notions in the context of MUMs, as shown by the following result. 

\begin{cor}\label{amumscor}
If $\mathcal P = \{P_a\}_{a=1}^d$ and $\mathcal Q = \{Q_b\}_{b=1}^d$ are an $\epsilon$-approximate MUM on $M_n$, then the algebras $\mathcal A$ and $\mathcal B$ they generate are $o(\sqrt{\epsilon})$-private.  
\end{cor}

\begin{proof}
By Proposition~\ref{aquasimums}, we have  $Q(\mathcal A, \mathcal B) \leq (1 + \epsilon)^2 = 1 + \epsilon (2+\epsilon)$. Thus we can apply Theorem~3.3 of \cite{kribs2021approximate} (using the fact that $d= \dim \mathcal A = \dim \mathcal B$) to conclude that $\mathcal A$ and $\mathcal B$ are $\epsilon^\prime$-private relative to each other, where $\epsilon^\prime$ is obtained via the equation 
$\epsilon^\prime = \sqrt{\epsilon (2+\epsilon)(d-1)},$    
and the result follows. 
\end{proof}

\begin{rem}
{\rm 
    It may be worthwhile to further investigate the privacy for MUMs and approximate MUMs, specifically with a focus on the conditional expectations $\mathcal E_{\mathcal A}$ and $\mathcal E_{\mathcal B}$, and the quasiorthogonality and approximate relative privacy of their algebras. 
    In fact, we note that MUMs were used in \cite{farkas2023MUMs} to construct superdense coding protocols, a key tool in quantum privacy and communication. We further note that if $\{ P_a \}_{a=1}^d$ is a family of projections on $\mathbb C^n$ with mutually orthogonal ranges of the same rank that span the whole space, then the trace-preserving conditional expectation $\mathcal E_{\mathcal A}$ of $M_n$ onto the commutative algebra $\mathcal A = \mathrm{span}\,\{ P_a \}$ that they generate has a particularly nice form: $\mathcal E_{\mathcal A} (X) = \frac1n \sum_{r=1}^d U_r X U_r^*$, where $U_r = \sum_{a=1}^{d} \omega ^{r(a-1)}P_a$, for $1\leq r \leq d$, are unitary operators that form an orthogonal basis (in the trace inner product) for $\mathcal A$.  
    In particular, the unitary operators $\{ U_r \}_{r=1}^d$ that implement the conditional expectations are exactly the families of unitaries that generate the superdense coding protocols in \cite{farkas2023MUMs}. 
}
\end{rem}

\section*{Acknowledgements}
D.W.K. was supported by NSERC Discovery Grant number RGPIN-2024-400160. R.P.\ was
supported by the NSERC Discovery Grant number RGPIN-2022-04149.  S.P.\ was supported by NSERC Discovery Grant number 1174582, the Canada Foundation for Innovation (CFI) grant number 35711, and the Canada Research Chairs (CRC) Program grant number 231250.


\end{document}